\documentclass[11pt]{amsart}

\usepackage{amsmath}
\usepackage{amssymb}
\usepackage{graphicx}

\newtheorem{theorem}{Theorem}[section]
\newtheorem{lem}[theorem]{Lemma}
\newtheorem{cor}[theorem]{Corollary}

\theoremstyle{definition}
\newtheorem{definition}[theorem]{Definition}
\newtheorem{example}[theorem]{Example}

\theoremstyle{remark}

\numberwithin{equation}{section}

\def\d{\delta}

\def\V{\Vert}

\def\d{\displaystyle}

\keywords{measurable function spaces, fixed point,  affine Lipschitzian mappings, equivalent norms, $\mu$-a.e. convergence, convergence in measure}
\begin{document}

\title[  Fixed points for affine mappings ]{ Koml\'os' Theorem and the Fixed Point  Property for affine  mappings}
\author{ T. Dom\'{\i}nguez Benavides, M. A, Jap\'on}

\address{Tom\'as Dom\'{\i}nguez Benavides\hfill\break
Departamento de  An\'{a}lisis Matem\'{a}tico, Universidad de
Sevilla, P:O: Box 1160, 41080-Sevilla, Spain}
 \email{tomasd@us.es}

 \address{M. Jap\'on \hfill\break
Departamento de  An\'{a}lisis Matem\'{a}tico, Universidad de
Sevilla, P:O: Box 1160, 41080-Sevilla, Spain}
 \email{japon@us.es}

\thanks{The   authors are partially supported by MINECO,
Grant MTM2015-65242-C2-1-P and Andalusian Regional Government Grant
FQM-127. }

\begin{abstract}

Assume that $X$ is a Banach space of measurable functions for which Koml\'os' Theorem holds. We associate to any  closed convex bounded subset $C$ of $X$   a coefficient $t(C)$ which attains its minimum value when $C$ is closed for the topology of  convergence in measure and   we prove some fixed point results for affine  Lipschitzian  mappings, depending on the value of $t(C)\in [1,2]$ and  the value of the Lipschitz constants of the iterates.  As a first consequence, for every $L<2$, we deduce the existence of fixed points for affine uniformly $L$-Lipschitzian mappings defined on the closed unit ball of $L_1[0,1]$. Our main theorem also provides a wide collection of convex closed bounded sets in $L^1([0,1])$ and in  some other spaces of functions,    which satisfy  the fixed point property for affine nonexpansive mappings. Furthermore, this property is still preserved by  equivalent renormings when the Banach-Mazur distance is small enough. In particular, we prove that the failure of the fixed point property for affine nonexpansive mappings in $L_1(\mu)$  can only occur in the extremal case $t(C)=2$. Examples are displayed proving that our fixed point theorem is optimal in terms of the Lipschitz constants and the coefficient $t(C)$.

\end{abstract}

\maketitle
\section{Introduction}

  Fixed point theory for nonexpansive mappings and Lipschitzian mappings has been widely developed in the last 40 years. In many Banach spaces, it is well known that any nonexpansive mappings $T$ defined from a convex closed bounded subset $C$ into $C$ has a fixed point. However, in some other spaces this assertion is false. It can be surprising that,
  as observed in \cite[Chapter 2]{handbook},  most  relevant examples about the failure of the fixed point property for nonexpansive mappings in closed convex bounded sets  involve affine mappings  (see also \cite{Lennardetalonc0, GK}). Of course, this failure cannot occur when $C$ is  weakly compact due to  Tychonov Fixed Point Theorem (and this fact  may well  be the reason for the absence of  examples  of nonexpansive fixed point free mappings defined on a closed convex bounded  subset of  a reflexive space). However, we will show that, in the case of   $L^1$-like spaces and due to Koml\'os' Theorem, the failure of the fixed point property for affine nonexpansive mappings can only occur in a very restricted class of sets

  In Section 1, we introduce some preliminary definitions and notation. In Section 2 we state and prove our main theorem. To do that, we associate to any  closed convex bounded subset $C$ of the space $X$   a coefficient $t(C)\in [1,2]$,  which  is equal to 1 when $C$ is closed for the topology of  convergence in measure and  it can be seen as a measure of the \lq\lq non-closedness'' of the set $C$ for this topology.  We prove some fixed point results for a  class of affine mappings which contains all uniformly $L$-Lipschitzian  mappings, depending on the value of $t(C)$, the value of $L$ and the  Opial modulus of the space with respect to the convergence in measure topology. Since a fixed point theorem for  $L$-Lipschitzian  mappings with $L\geq 1$ immediately yields to the existence of fixed points for nonexpansive mappings, our result can be applied to this class of mappings and we  obtain  stability results for  the fixed point property for affine nonexpansive mappings under renormings.   We show a wide collection of closed convex sets in $L^1([0,1])$ which satisfy the fixed point property for affine uniformly $L$-Lipschitzian  mappings (including its  closed unit ball if $L<2$) and we show some applications of our main theorem in some other spaces as Orlicz spaces or non-commutative $L^1$-spaces. In particular, we prove that the failure of the fixed point property for affine nonexpansive mappings in $L_1(\mu)$  can only occurs in the extremal case $t(C)=2$.

  Since every nonreflexive space contains a closed convex bounded set which fails the fixed point property for affine continuous mappings \cite{milman}, we cannot expect the validity of our results without any Lipschitz restriction. In fact, in Section 3, we   include some examples showing that the value of the constant $L$  in our main theorem is optimal for  all possible values of $t(C)$.

  It is worth noting that, while standard fixed point results for nonexpansive mappings in $L_1(\mu)$   assume compactness for the set  $C$ with respect to  the topology of convergence in measure, due to our affinity assumption,  we do not need any compactness assumption for $C$.

\section{Preliminaries}

Given a Banach space $(X,\V\cdot\V)$ endowed with a linear topology $\tau$, it is said that $X$ has the non-strict Opial condition with respect to (w.r.t.) $\tau$ if for every sequence $(x_n)$ which is $\tau$-convergent to some $x_0\in X$
$$
\liminf_n\V x_n-x_0\V\le\liminf_n\V x_n-x\V
$$
for every $x\in X$. In case that the above inequality is strict for every $x\ne x_0$, it is said that $X$ verifies the Opial condition w.r.t. $\tau$.
Associated to the Opial property the following coefficient is defined for every $c\ge 0$:
$$
r_\tau(c)=\inf\{\liminf_n\V x_n-x\V-1\},
$$
where the infimum is taken over all $x\in X$ with $\V x\V\ge c$ and all $\tau$-null sequences $(x_n)\subset X$ with $\liminf_n\V x_n\V\ge 1$ (see for instance \cite[Chapter 4]{handbook}). The modulus $r_\tau(\cdot)$ is a non-decreasing and continuous function in $[0,+\infty)$ \cite[Theorem 3.5, page 103]{ADL}. It is clear that $r_\tau(c)\ge 0$ for every $c\ge 0$ whenever  $X$ verifies the non-strict Opial condition w.r.t. $\tau$. In case that $r_\tau(c)>0$ for every $c\ge 0$, the Banach space is said to satisfy the uniform Opial condition w.r.t. $\tau$.

\medskip

We  recall Koml\'os' Theorem:

\begin{theorem}\cite{Kom}
Let $\mu$ be a probability measure. For every bounded sequence $(f_n)$ in $L_1(\mu)$, there exists a subsequence $(g_n)$ of $(f_n)$ and a function $g\in L_1(\mu)$ such that for every further subsequence $(h_n)$ of $(g_n)$,
$$
{1\over n}\sum_{i=1}^n h_i\to g \ \mu-a.e.
$$
\end{theorem}

 Koml\'os' Theorem has been extended to a broader class of Banach function spaces.

\begin{definition}
We say that a Banach function space $X$ associated to a  $\sigma$-finite measure satisfies the Koml\'os' condition if    for every bounded sequence $(f_n)\subset X$ there exists a subsequence $(g_n)$ of $(f_n)$ and a function $g\in X$ such that for every further subsequence $(h_n)$ of $(g_n)$,
$$
{1\over n}\sum_{i=1}^n h_i\to g \ \mu-a.e.
$$
\end{definition}

It is proved in \cite{len}  that every Banach function space $X$ over  a $\sigma$-finite complete measure space $(\Omega,\Sigma,\mu)$, such that $X$ is weakly finitely integrable and has the Fatou property  (see definitions in \cite{len}) satisfies  Koml\'os' condition.   Among this class,  $L_p(\mu)$ ($1\le p\le +\infty$), Lorentz, Orlicz and Orlicz-Lorentz spaces are included.

\vskip .3cm

\section{Main result}

We start this section introducing  the following geometric coefficient which will be essential in the proof of our main theorem.

\begin{definition} Let $X$ be a Banach space endowed with a topology $\tau$. Let $C$ be a norm-closed convex bounded subset of $X$ which contains some $\tau$-convergent sequences. We define
$$
t(C)=\inf\left\{ \lambda\ge 0 : \d{\inf_{c\in C} }\limsup_n\V c- x_n\V\le \lambda \limsup_n\V x- x_n\V\ \right\}
$$
where  $(x_n)$ and $x$ run over all sequences  $(x_n)\subset C$  with $\tau-\lim x_n=x\in X$.  \end{definition}

 Note that  $t(C)\ge 1$ whenever
 $X$ verifies the non-strict Opial condition and $t(C)\leq 2$ for any case.
If $X$ is a Banach space satisfying the uniform Opial condition w.r.t. $\tau$,  it is not difficult to check that $t(C)=1$ if and only if $C$ is $\tau$-closed.
Since $t(C)$ attains its minimum value when $C$ is $\tau$-closed, the coefficient  $t(C)$ can be understood as a measure of the\lq\lq non-closedness'' of $C$ for the topology $\tau$. If we denote  by $H(C,\overline{C}^\tau)$   the Hausdorff distance between a set $C$ and its $\tau$-closure, the previous idea
 can be illustrated with the following  lemma:

  \begin{lem} \label{lema}
Let  $C$ be a norm-closed convex bounded subset of $L_1([0,1])$ and $\tau$ the  topology  of the convergence  in measure.
 Then  $H(C,\overline{C}^\tau)\leq diam (C)/2.$ The extremal case, $H(C,\overline{C}^\tau)= diam (C)/2$, implies that $t(C)=2$.
\end{lem}
\begin{proof}
It is well known that for every $z\in L_1([0,1])$ and every  sequence  $\tau$-null sequence $(x_n)$  we have
$$
\limsup_n\V x_n+z\V=\limsup_n\V x_n\V+\V z\V.\qquad \qquad (\ast)
$$
If $(x_n)\subset C$ with $\tau-\lim_n x_n=x$,
  the above equality implies
     \begin{eqnarray*}diam (C) &\geq& \limsup_m\limsup_n\Vert (x_n-x)+(x-x_m)\Vert\\&=&\limsup_n\Vert x_n-x\Vert+\limsup_m\Vert x-x_m\Vert =2 \limsup_n\Vert x_n-x\Vert,\end{eqnarray*}
 and consequently $\limsup_n\Vert x_n-x\Vert \leq diam (C)/2$.
The definition of the Hausdorff distance now implies   $H(C,\overline{C}^\tau)\leq diam (C)/2.$

\medskip

Assume $H(C,\overline{C}^\tau)= diam (C)/2$.
For every $\epsilon>0$ there exists $x\in \overline{C}^\tau$ such that $d(x,C)\geq diam(C)/2-\epsilon$. Since the topology $\tau$ is metrizable, we can choose a sequence $(x_n)$ in $C$ which is  $\tau$-convergent to $x$.  For every $c\in C$ we have
\begin{eqnarray*} \limsup_n \Vert x_n-c\Vert &=&\limsup_n \Vert x_n-x\Vert +\Vert x-c\Vert \\ &\geq& \limsup_n \Vert x_n-x\Vert + diam (C)/2-\epsilon\\&\geq & 2\limsup \Vert x_n-x\Vert -\epsilon,\end{eqnarray*}
which implies that $t(C)=2$.
\end{proof}







\vskip 0.3cm

Throughout this section, we consider that  $X$ is a function Banach space over a finite or $\sigma$-finite measure space with the Koml\'os' condition. In case that the measure is finite we can endow $X$ with the topology of the convergence in measure. Convergence in measure for a finite measure  is related  to a.e. convergence as follows: For $(f_n)$  a sequence converging to $f$ in measure, there is a subsequence $(f_{n_k})$ which converges to $f$ almost everywhere. Conversely, if $(f_n)$ tends to $f$ a.e. then $(f_n)$ converges to $f$  in  measure \cite[pages 156-158]{HS}.  The same holds for a $\sigma$-finite measure  and the local convergence in measure.
Due to the Koml\'os' condition, every norm-closed convex bounded subset of $X$ contains a sequence which is convergent in measure or locally in measure. In fact, due to these facts, we can equivalently define the coefficient $t(C)$ by
$$
t(C)=\inf\{\lambda\ge 0: \inf_{c\in C} \limsup_n\V c- x_n\V\le \lambda \limsup_n\V x- x_n\V:\ x_n\to x\  \mu-{\rm a.e.}\ \}
$$
and for the Opial modulus $r_\tau(\cdot)$ we can replace the $\tau$-convergence by convergence $\mu$-a.e.

\medskip

If $T:C\to C$  is a Lipschitzian mapping,  we denote by $|T|$ its exact Lipschitz constant, that is
$$
|T|=\sup\left\{ {\V Tx-Ty\V\over \V x-y\V}: \ x,y\in C, \ x\ne y\ \right\}
$$
and  we define
$$
S(T)=\liminf_n {|T|+\cdots+|T^n|\over n}.
$$
A sequence $(x_n)\subset C$ is an approximate fixed point sequence (a.f.p.s.)  for $T$ whenever
$\lim_n\V x_n-Tx_n\V=0$ and
it is clear that this property is inherited by all its subsequences. The main result of the article is the following:

\begin{theorem}\label{main} Let $X$ be a function Banach space with the Koml\'os' condition  and satisfying the non-strict Opial property. Let $T:C\to C$ be an affine Lipschitzian mapping.
If
$$
\displaystyle{ S(T)<{1+r_X(1)\over t(C)}}
$$
then $T$ has a fixed point.
\end{theorem}

\begin{proof}
 Since $T$ is affine,  it can be checked that for every $x\in C$ the sequence
$$
x_n:={Tx+T^2x+\cdots +T^nx\over n}
$$
is an a.f.p.s. and   the sequence of the arithmetic means of an a.f.p.s. is an a.f.p.s. as well. Applying Koml\'os'  condition, there exists a subsequence of $(x_n)$  and $x\in X$,  such that for all further subsequences,  the  sequence   of   successive
arithmetic means  converges  to $x$ $\mu$-a.e. Consequently, we can always assume that the  set
$$
\mathcal{D}(C)=\{\{(x_n),x\}:\ (x_n) \mbox { is an a.f.p.s. in $C$ and } \ \lim_n x_n=x \ \mu-{\rm a.e.}  \}
$$
is nonempty.

For every $y\in C$ we define
$$
r(y)=\inf\{\limsup_n\V y-x_n\V: \{(x_n),x\}\in\mathcal{D}(C)\  \}.
$$

We first prove that $T(y)=y$ whenever $r(y)=0$.
Indeed, let $\epsilon>0$ and take $\{(x_n),x\}\in\mathcal{D}(C)$ with
$
\limsup_n \V y-x_n\V\le \epsilon.
$
Then
$$
\begin{array}{lll}
\V Ty-y\V & \le & \limsup_n\V Ty-Tx_n\V + \limsup_n\V Tx_n-x_n\V+\limsup_n\V x_n-y\V\\
          &\le & |T|\limsup_n\V y-x_n\V +\limsup_n\V x_n-y\V\le (|T|+1)\epsilon
          \end{array}
          $$
           and $T(y)=y$ since $\epsilon$ is arbitrary. Thus, our target will be to find some $y\in C$ with $r(y)=0$.

 To do that,  choose $\epsilon>0$ such that
 $$
 S(T)<{1+r_X(1)\over t(C)}{1-\epsilon\over (1+\epsilon)^2}.
 $$ and   an arbitrary
 $x_0\in C$.
 If $r(x_0)>0$, take $\{(x_n),x\}\in\mathcal{D}(C)$ with
 $$
 \limsup_n\V x_0-x_n\V <r(x_0)(1+\epsilon).
 $$

We denote by $\phi_{(x_n)}(\cdot)$ the convex function
 $$
 \phi_{(x_n)}(y)=\limsup_n\V y-x_n\V,\qquad y\in X.
 $$

It can be easily checked that  $\lim_n \V T^s x_n-x_n\V=0$ for every $s\in \mathbb{N}$. Hence
 $$
\begin{array}{lll}
 \phi_{(x_n)}(T^s x_0)& = &\limsup_n\V T^s x_0-x_n\V=\limsup_n \V T^sx_0-T^sx_n\V\\
                      & \le &  |T^s| \limsup_n \V x_0-x_n\V = |T^s|\phi_{(x_n)}(x_0) .
                      \end{array}$$

Therefore, if we define the sequence
 $$
 z_s:={Tx_0+\cdots+T^s x_0\over s}
 $$
we know that $(z_s)$ is an a.f.p.s and for every $s\in\mathbb{N}$
$$
\displaystyle{\phi_{(x_n)}(z_s)\le {|T|+|T^2|+\cdots +|T^s|\over s} \ \phi_{(x_n)}(x_0)}.
$$
Taking limits
$$
\liminf_s \phi_{(x_n)}(z_s)\le S(T)\  \phi_{(x_n)}(x_0).
$$

Applying Koml\'os'  condition, we can find a subsequence $(z_{s_i})$  and $z\in X$,  such that $\lim_i\phi_{(x_n)}(z_{s_i})=\liminf_s \phi_{(x_n)}(z_s)$ and
$$
{\bar z}_p={z_{s_1}+\cdots + z_{s_p}\over p}
$$
converges  to $z$ $\mu$-a.e. Moreover, by convexity and taking limits  we have
 $$
\begin{array}{lll}\d{ \limsup_p \phi_{(x_n)}({\bar z}_p)}& \le & \d{\limsup_p{ \phi_{(x_n)}(z_{s_1})+\cdots+ \phi_{(x_n)}(z_{s_p})\over p}}\\
    &= &\d{ \lim_i\phi_{(x_n)}(z_{s_i})}\le S(T)\  \phi_{(x_n)}(x_0).\\
    \end{array}
 $$
 Consequently, we have obtained an a.f.p.s. $({\bar z}_p)$, convergent  to some $z\in X$ $\mu$-a.e. and with
  $$
 \limsup_p\limsup_n\V \bar{z}_p-x_n\V < S(T)\  r(x_0)(1+\epsilon).
 $$

Define
$$
\rho:= {r(x_0)(1-\epsilon)\over   t(C)(1+\epsilon)}.
$$
We claim that
$$
\min\{ \limsup_n\V x-x_n\V, \limsup_p\V z-{\bar z}_p\V\} \le \rho.
$$
Indeed, otherwise, using the Opial modulus
$$
\limsup_p\left\V{{\bar z}_p-x\over \rho}\right\V = \limsup_p\left\V{{\bar z}_p-{ z}+{ z}-x\over \rho}\right\V\ge 1+r_\tau\left({\V {z}-x\V\over \rho}\right)\ge 1
$$
and
$$
\begin{array}{lll}
S(T)\ r(x_0)(1+\epsilon) & > & \d{\limsup_p\limsup_n\V {\bar z}_p-x_n\V }\\
                          &=  & \d{\rho \limsup_p\limsup_n\left\V{{\bar z}_p-x\over \rho}+{x-x_n\over \rho}\right\V}\\
                        &\ge &\d{ \rho \limsup_p \left[1+ r_\tau\left({\V {\bar z}_p-x\V\over \rho}\right)\right]}\\
                        &=   & \d{\rho\left[1+ r_\tau\left({\limsup_p\V {\bar z}_p-x\V\over \rho}\right)\right]}
                         \ge  \rho [1+r_\tau(1)],
                        \end{array}
                        $$
 which is a contradiction with the choice of $\epsilon$ and the definition of $\rho$.
\vskip 0.3cm

According to the definition of the coefficient $t(C)$, there exist some ${\bar x}, {\bar z}\in C$ such that
$$
\limsup_n \V \bar{x}-x_n\V\le t(C)(1+\epsilon) \limsup_n\V x-x_n\V,
$$
$$
\limsup_n \V \bar{z}-{\bar z}_n\V\le t(C)(1+\epsilon) \limsup_n\V z-{\bar z}_n\V.
$$

\vskip 0.3cm

If $ \limsup_n\V x_n-x\V\le \rho$ then
 $$
  \limsup_n\V x_n-{\bar x}\V\le r(x_0)(1-\epsilon)
  $$
 which implies $r(\bar{x})\le r(x_0)(1-\epsilon)$. Note
 $$
 \begin{array}{lll}
 \V \bar{x}-x_0\V &\le & \limsup_n\V x_n-x_0\V+\limsup_n\V x_n-{\bar x}\V\le r(x_0)(1+\epsilon)+r(x_0)(1-\epsilon)\\
                 & = & 2r(x_0).
                 \end{array}
 $$
On the other hand, if   $ \limsup_p\V {\bar z}_p-z\V\le \rho$, then
 $$
  \limsup_p\V {\bar z}_p-{\bar z}\V\le r(x_0)(1-\epsilon)
  $$
 which  implies  $r({\bar z})\le r(x_0)(1-\epsilon)$. In this case,
 $$
\begin{array}{lll}
 \V \bar{z}-x_0\V&\le& \limsup_n\V x_0-x_n\V+\limsup_p\V \bar{z}-\bar {z}_p\V\\
                 & +&\limsup_p\limsup_n\V {\bar z}_p-x_n\V\\
           & \le & (1+\epsilon)r(x_0)+ r(x_0)(1-\epsilon)+ (1+\epsilon)S(T) r(x_0)\\
             & = & [2+(1+\epsilon)S(T)]r(x_0).
           \end{array}
 $$
 In every case, we have obtained some $w(x_0)\in C$ with
 $$
 r(w(x_0))\le r(x_0)(1-\epsilon)\ {\rm and }\ \ \V x_0-w(x_0)\V\le [2+(1+\epsilon)S(T)]r(x_0)
 $$
Now we construct the sequence $a_1=x_0$, $a_{n+1}=w(a_n)$ for $n\geq1$. Since $(a_n)$ is a Cauchy sequence, there exists $a=\lim_na_n$ such that $r(a)=0$ and $T$ has a fixed point.

\end{proof}

Equality $(\ast)$ used  in the proof of Lemma \ref{lema}  for $L_1([0,1])$ also holds for $L_1(\mu)$ for $\mu$ either a finite or $\sigma$-finite measure and implies that $1+r_{\tau}(1)=2$ when   $X=L_1(\mu)$. Consequently we can  deduce:

\begin{cor} \label{cor} Let $C$ be a norm-closed convex bounded subset of $L_1(\mu)$ and $T:C\to C$ be an affine Lipschitzian mapping. Then $T$ has a fixed point whenever
$$
S(T)<{2\over t(C)}.
$$
\end{cor}

In particular, if we consider any convex bounded subset  of $L_1(\mu)$ which is  $\tau$-closed, we can derive the existence of a fixed point for every affine Lipschtizian mapping with $S(T)<2$. An example of such a set is  $B_{L_1(\mu)}$, the closed unit ball of $L_1(\mu)$. Furthermore, the failure of the fixed point property for affine nonexpansive mappings in closed convex subsets of $L_1([0,1])$ (or $\ell_1$  which is isometrically embedded in $L_1([0,1])$) can only occur when $t(C)$ attains its maximum value 2.
Note that the affinity condition can not be dropped  because there exist some nonexpansive mappings  from $B_{L([0,1])}$ into itself without fixed points.
In fact,
 for every norm-closed convex bounded set $C$ of $L_1([0,1])$ containing a closed interval, there is  a fixed point free nonexpansive mapping $T:C\to C$ \cite{DLT} (see also \cite[Chapter 2]{handbook} and \cite{Sine}).

\medskip

We next show  further examples of norm-closed convex bounded subsets of $L^1([0,1])$   where Theorem \ref{main}  can be applied:

\begin{example} \label{EX1} We use the following notation:
$$
C:=\{f\in L_1([0,1]):\ f\ge 0,\ \int_0^1 f(t)dt=1\ \},
$$
$$
C_a:={\rm co}(C\cup\{a\}),\qquad a\in [0,1].
$$
All these sets are norm-closed, convex and bounded. They are not weakly compact because
they contain the sequence $\{ n\chi_{[0,1/n]}\}$ which has no weakly convergent
subsequence. Furthermore, no one is contained in a compact set for the topology of the convergence in measure,  since they
contain the sequence  $\{ 1+r_n\}$ (where $\{r_n\}$ is the  Rademacher sequence) which has no a.e. convergent subsequence. In
particular, we cannot deduce   existence of fixed points for these sets by using any known fixed point theorem
 for  compact in measure sets, as in \cite{DBGFJP}.
We will prove that $t(C_a)=1+a$ for every $a\in [0,1]$:

Let $\{g_n:=\lambda_n f_n+(1-\lambda_n) a\}$  be a sequence in $C_a$ which is
convergent to some $g$ a.e., where $(\lambda_n)\in [0,1]$ and $(f_n)\subset C$.
Throughout a subsequence, if necessary, we can assume that $\lambda_n$ converges to some $\lambda
\in [0,1]$, which implies that $\{f_n\}$ converges to a function $f\ge 0$ a.e. and
$g=\lambda f+ (1-\lambda)a$.
From  Fatou's Lemma we obtain  $\Vert
f\Vert\leq 1$.  We have
\begin{eqnarray*}\limsup_n \Vert g_n-g\Vert&=& \limsup_n\Vert \lambda_nf_n+(1-\lambda_n)a-\lambda
f-(1-\lambda) a\Vert\\
&=&
 \lambda (\limsup_n \Vert f_n-f\Vert )=\lambda( \limsup_n\Vert f_n\Vert -\Vert
f\Vert)\\
&=&\lambda(1-\Vert f\Vert).\end{eqnarray*}
Furthermore, if $f\not= 0$ we have that the function $$f+(1-\Vert f\Vert)a=\Vert
f\Vert \left( \frac{f}{\Vert f\Vert}\right) +(1-\Vert f\Vert)a$$
belongs to $C_a$. The same is true if $f=0$. Denote  $h=(1-\lambda)a+\lambda (f+(1-\Vert f\Vert)a)\in C_a$. We have
$$
\begin{array}{lll}
\limsup_n\V g_n-h\V&=& \limsup_n \V g_n-g\V+\V g-h\V\\
                             & = & \lambda(1-\Vert f\Vert)+  \lambda(1-\Vert f\Vert)a\\
                             & = &(1+a)\limsup_n\V g_n-g\V.
\end{array}
$$
Thus, $t(C_a)\leq 1+a$. The equality is a consequence of the fact that the sequence $f_n=\{ n\chi_{[0,1/n]}\}$ converges to 0 a.e. and $\limsup_n \Vert f_n-f\Vert\geq 1+a $ for every $f\in C_a$.
Therefore, for every $a\in [0,1]$ and $T:C_a\to C_a$ affine Lipschitzian with $S(T)<{2\over 1+a}$, the set of fixed points is nonempty.
\end{example}

Besides the function spaces $L_1(\mu)$, there are some broader classes of function Banach spaces with the uniform Opial property with respect to the almost everywhere convergence and satisfying the Koml\'os' condition. Indeed, in case of Orlicz function spaces for an Orlicz function $\Phi$ satisfying the $\Delta_2$-condition, it was proved in \cite[Theorem 3]{DJ} that the Orlicz space $X=L_\Phi(\mu)$ verifies the uniform Opial condition w.r.t. the convergence almost everywhere and
$$
1+r_{\tau}(1)\ge a\left({1\over 2}\right)
$$
where the function $a(\cdot)$ was defined by
$$
a(\delta)=\inf\left\{{\Phi^{-1}(t)\over \Phi^{-1}(\delta t)}:\ t>0\ \right\},\qquad \forall \delta>0.
$$
Thus we can conclude:

\begin{cor}
Let $\Phi$ be an Orlicz function satisfying the $\Delta_2$-condition and $X=L_\Phi(\mu)$ endowed with the Luxemburg norm. Let $C$ be a norm-closed convex bounded set  and $T:C\to C$ an affine Lipschitzian mapping with
$$
S(T)<{a(1/2)\over t(C)}.
$$
Then $T$ has a fixed point.
\end{cor}

A similar result can be obtained for Orlicz function spaces endowed with the Orlicz  norm in case that $\Phi$ is an $N$-function \cite[Theorem 5]{DJ}.

\medskip
To finish this section, we extend our results to non-commutative $L_1$-spaces associated to a finite von Neumann algebra. Note that for every $\sigma$-finite measure space,  the Banach space $L_\infty(\mu)$ is a finite von Neumann algebra and the corresponding
$L_1(\mu)$ is  a particular example of a non-commutative (in fact, commutative) $L_1$-space. For  standard notation and some background on non-commutative $L_1$-spaces the reader can consult for instance \cite{E-Nelson, G-Pisier&Q-Xu}.

\medskip

Let $(\mathcal{M},\tau)$ be a finite von Neumann algebra, and consider $X=L_1(\mathcal{M},\tau)$ endowed with the usual norm
$$
\V x\V= \tau(|x|).
$$
The measure topology is defined via the following fundamental system of neighborhoods of zero: for every $\epsilon,\delta>0$  let
$$
N(\epsilon,\delta)=\{x\in\mathcal{M}: \exists \ p \  \hbox{ projection in $\mathcal{M}$ with } \V xp\V_\infty\le \epsilon \ {\rm and }\ \tau(p^\perp)\le \delta\}.
$$
In case that we consider $L_1([0,1])$ and the trace given by $\tau(f)=\int_0^1 f dx$, the previous topology
 coincides with the usual topology of the convergence in measure. Note that $L_1(\mathcal{M},\tau)$ is a $L$-embedded Banach space and the measure topology is an abstract measure topology in the sense of \cite{Pfi}. Thus,
 equality $(\ast)$ given in the proof of Lemma \ref{lema} can be generalized to the frame of $L_1(\mathcal{M},\tau)$ and the measure topology (see \cite{mangeles,H-Pfitzner, Pfi}) and Koml\'os' condition is extended in \cite[Proposition 3.11]{R}.
As a consequence of Theorem \ref{main} we can conclude:

\begin{cor}
Let $(\mathcal{M}, \tau ) $ be a finite von Neumann algebra and $C$ be a norm-closed convex bounded subset of $L_1(\mathcal{M},\tau)$. Every affine Lipschitzian mapping $T:C\to C$ has a fixed point whenever
$$
S(T)<{2\over t(C)}.
$$
In case that $C$ is closed in measure, as the close unit ball, $T$ has a fixed point if $S(T)<2$.
\end{cor}


\section{Some sharp examples}

In this section, we show that the statement of  Theorem \ref{main}  is sharp for every possible value of $t(C)$. First, we check that either the  condition $t(C)=2$  or $s(T)=2$ does not imply the existence of fixed points in $L_1([0,1])$.

\begin{example}\label{failure} Let $C$ be the subset of $L_1([0,1])$ given in Example \ref{EX1}.   For every $f\in L_1([0,1])$ we define  $T(f)(t)=2f(2t)$, where we
assume $f(t)=0$ if $t>1$. Note that $T$ is an affine isometry, $T(C)\subset C$ and supp $T^n(f)\subset [0,1/2^n]$ for every $n\in\mathbb{N}$. This implies that $f=0$ a.e., which does not belong to $C$, is the unique possible fixed point for $T$.   Consequently, $T:C\to C$ fails to have a fixed point and  $t(C)$ must be equal to 2.
\end{example}

\begin{example}

We consider the set $C_0$ introduced in  Example \ref{EX1} which verifies $t(C_0)=1$. Define $G: C_0\to C_0$ by $G=TR$, where $T$ is the mapping in Example \ref{failure} and
 $R:C_0\to C$ is given  by
 $$R(f)=(1-\Vert f\Vert)+f.$$

 Note that $\Vert
\lambda f+(1-\lambda)g\Vert =\lambda \Vert f\Vert +(1-\lambda)\Vert g\Vert$ for
every $f,g\in C_0$ which implies that  $R$ is an affine  retraction. Furthermore $G$ is fixed point free  because  $T$ is. Since $(TR)^n=T^nR$
and  $\vert R\vert \leq 2$, $S(G)\le 2$.
Finally,  Corollary \ref{cor} implies  $S(G)= 2$  due to the absence of fixed point.




\end{example}

Finally,   for every possible value of $t(C)\in (1,2)$, we next show an example  of an affine Lipschitzian  mapping,   failing to have any fixed point, and with  $S(T)={1+r_\tau(1)\over t(C)}$. That is, Theorem \ref{main} is sharp in every possible case.

\begin{example}
Take $(g_n)$ a sequence of normalized functions in $L_1([0,1])$  supported on a  pairwise disjoint sequence of subsets of $[0,1]$. For every $t\in (1,2)$ we consider the following norm-closed convex bounded  set
$$
C_t:=\left\{ s_1(t-1) g_1+\sum_{n=2}^\infty s_n g_n: \ \ s_n\ge 0, \sum_{n=1}^\infty s_n=1\ \right\}.
$$
 Define $T:C_t\to C_t$  by
$$
T\left(s_1(t-1) g_1+\sum_{n=2}^\infty s_n g_n\right)=\sum_{n=1}^\infty s_n g_{n+1}.
$$
 It is not difficult to check that $T$ is fixed point free,  affine and
$$
\V T^nf-T^ng\V\le {2\over t} \V f-g\V
$$
for every $f,g\in C_t$. For  $f=(t-1)g_1$ and $g=g_2$ we have
$$
\V T^nf-T^ng\V=\V g_{n+1}-g_{n+2}\V=2={2\over t}\V f-g\V
$$
which implies that $S(T)={2\over t}$. Now let us check that $t(C_t)=t$. Indeed, let $(f_n)$ be a sequence in $C$ which converges a.e. to some $f\in L_1([0,1])$. In particular $f=s_1(t-1)f_1+\sum_{n=2}^\infty s_n f_n$ where $\delta_f:=\sum_{n=1}^\infty s_n\le 1$ and
$$
\limsup_n\V f_n-f\V=\limsup_n\V f_n\V-\V f\V = 1-\delta_f.
$$

Let $g=f+(1-\delta_f)(t-1)f_1\in C_t$. Then
$$
\begin{array}{lll}\inf_{h\in C_t}\limsup_n\V f_n-h\V& \le& \limsup_n\V f_n-g\V\\
      &= &\limsup_n\V f_n-f\V+(1-\delta_f)(t-1)\\
  &= &t \limsup_n\V f_n-f\V,
  \end{array}
$$
which shows that $t(C_t)\le t$. Consider the sequence $(g_n)_{n\ge 2}\subset C_t$ which tends to zero a.e. It can be checked that for every $h\in C_t$, $\limsup_n \V g_n-h\V\ge t$ and this proves that $t(C_t)=t$ for every $t\in (1,2)$.

\end{example}

\end{document}